\documentclass[review,10pt]{elsarticle}
\usepackage{lipsum}
\usepackage{amsthm}
\makeatletter
\def\ps@pprintTitle{%
 \let\@oddhead\@empty
 \let\@evenhead\@empty
 \def\@oddfoot{}%
 \let\@evenfoot\@oddfoot}
\makeatother
\usepackage[margin=2.5cm]{geometry}
\usepackage{lineno,hyperref}
\usepackage{amsmath,amssymb,latexsym, amscd}
\usepackage{exscale, eps fig, graphics}
\usepackage{array}
\usepackage{xcolor}
\usepackage{float}
\restylefloat{table}
\newcommand\numberthis{\addtocounter{equation}{1}\tag{\theequation}}
\usepackage{algorithmic}
\usepackage{placeins}
\usepackage{longtable}
\usepackage{makecell}
\usepackage[utf8]{inputenc}
\usepackage[justification=centering]{caption}
\usepackage{graphicx}
\usepackage{caption}
\usepackage{subcaption}
\usepackage{mathtools}
\hypersetup{pdfauthor={Name}}
\numberwithin{equation}{section}
\theoremstyle{plain}
\newtheorem{thm}{Theorem}[section]
\newtheorem*{thm*}{Theorem}

\newtheorem{lem}[thm]{Lemma}

\newtheorem{cor}{Corollary}
\newtheorem*{cor*}{Corollary}
\theoremstyle{definition}

\theoremstyle{remark}
\newtheorem{rem}{Remark}[section]

\def\i{\iota}
\def\Z{\mathbb Z}
\def\N{\mathbb N}

\def\R{\mathbb R}
\def\C{\mathbb C}

\def\I{{\mathfrak I}}

\def\Im{\operatorname{Im}}

\def\log{\operatorname{log}}

\newcommand{\BigO}[1]{\ensuremath{\operatorname{O}\left(#1\right)}}

\makeatletter
\def\pmod#1{\allowbreak\mkern10mu({\operator@font mod}\,\,#1)} 

\begin{document}


\title{Pair correlation of real-valued vector sequences}

\author[mymainaddress]{Sneha Chaubey}
\ead{sneha@iiitd.ac.in}

\author[mymainaddress]{Shivani Goel}
\ead{shivanig@iiitd.ac.in}

\address[mymainaddress]{Department of Mathematics, IIIT Delhi, New Delhi 110020}

\begin{abstract}
In this article, we investigate the fine-scale statistics of real-valued arithmetic sequences. In particular, we focus on real-valued vector sequences and show the Poissonian behavior of the pair correlation function for certain classes of such sequences, thereby extending previous works of Boca et al. and the first author on local statistics of integer-valued and rational-valued vector sequences.
  
\end{abstract}

\maketitle
\section{Introduction and Main Results}\label{sec1}
A sequence of real numbers $(x_n)_{n\ge 1}$ is said to be uniformly distributed (equidistributed) modulo one if for all intervals $I\in [0,1)$, we have
\[\frac{1}{N}\#\{n\le N: \{x_n\}\in I\}\to \text{length}(I) \ \ \ \text{as} \ {N\to \infty},\] where $\{x_n\}$ denotes the fractional part of $x_n$. 
We refer \cite{KN74} for the general theory of uniform distribution. One can study other finer distributions of a sequence, that is, distribution properties at the scale $\approx 1/N$. One way to study fine-scale properties of sequences is using the pair correlation statistic. For a real-valued sequence $(x_n)_{n\ge 1}$, any integer $N\ge 1$, and  every fixed interval $I=[-s,s]\subset \R$, the pair correlation function $R_2$ is defined as 
\[R_2(I,(x_n)_{n\ge 1},N)=\frac{1}{N}\#\left\{1\le m\ne n\le N: |x_m-x_n|\in \frac{I}{N}\right\}.\numberthis\label{paircor}\]
 If $R_2(I,(x_n)_{n\ge 1}, N)\to 2s$ as $N\to \infty$, the sequence $(x_n)_{n\ge 1}$ has Poissonian pair correlation (PPC) provided the limit exists. Poissonian pair correlation behavior of a sequence is related to its equidistribution property. A sequence having PPC must be uniformly distributed, but the converse is not necessarily true. See \cite{ALP18, GL17, Mar20}. The concept of PPC originally came from quantum physics. In particular, this is related to the distribution of the discrete energy spectrum of a Hamiltonian operator of a quantum system. We refer \cite{energyspectra} and the references cited therein for more details. Proving the PPC property for deterministic sequences is difficult; thus, there are few results in this direction. For example, in \cite{elbaz}, authors proved that the sequence ${\sqrt{n}}$ has PPC. In \cite{rudsar01}, it has been conjectured 
 that the sequence  $(\{\alpha n^2\})_{n\ge 1}$ will have Poissonian distribution provided $\alpha$ is Diophantine. Partial results in this direction can be found in \cite{Heath,MY, rudsar01, true}.
On the other hand, there are in literature considerable metric results involving averaging over all $\alpha$. The metric theory of pair correlation was studied from a mathematical point of view by Rudnick and Sarnak \cite{rudsar} in the late 20th century. They showed that for almost all $\alpha$, the sequence $({\alpha a(n)})$ modulo 1 has PPC, where $a(n)$ is a polynomial of degree at least two. Since then, several authors have studied the pair correlation property \cite{asit2017,bocarational,chaubey15,rudsar01,rudzah99,rudzah02} for integer or rational-valued sequences.

For real-valued sequences, more needs to be explored. In \cite{rudtec}, authors show that for almost all $\alpha$, $(\{\alpha x_n\})_{n\ge 1}$ ($x_n$ is a real number) has PPC. This involved counting solutions of certain Diophantine inequalities and generalization of the method in \cite{rudzah99}. It was shown that lacunary real-valued sequences have PPC. It turns out that it is difficult to find other number-theoretic sequences satisfying PPC using their hypothesis; as a result, Asitleitner et al. \cite{aistletner21} developed a new method involving the additive energy of sequences. Other than this, some mathematicians studied the pair correlation for the sequence $(\{n^{\theta}\})_{n\ge 1}$ where $\theta\in \R_{>0}$. In \cite{rudsar}, Rudnick and Sarnak proved that for $\theta>2$ and for almost all $\alpha$, the sequence  $(\{\alpha n^{\theta}\})_{n\ge 1}$ has PPC. Recently, in \cite{aistletner21} and \cite{rudtec21} authors extended the result in \cite{rudsar} for $\theta>1$, and for $0<\theta<1$, respectively. In \cite{lutsko2021pair}, authors showed that for $0<\theta<1/3$, the pair correlation of $(\{\alpha n^{\theta}\})_{n\ge 1}$ is Poissonian for all $\alpha.$
 Recently, in \cite{radziwill}, this result has been improved for $0<\theta<1/3+0.0341...$. 
 
The general interest in the distribution of sequences via the pair correlation property, as pointed out in the references cited above, justifies this topic's further study. The principal goal of the present article is to provide new examples of real-valued sequences behaving like that of a random sequence in the setting of vector sequences. To this effect, we consider real-valued vector sequences. For a fixed integer $r\ge 1$, a vector sequence is an injective map $\vec a:\N^r \to \R^r$. For this sequence, an arbitrary but fixed interval $I=[-s,s]$, any integer $N\ge 1$, and $\vec \alpha\in \R^r$, the pair correlation function is defined as
\[R_2(I,\vec a,N)(\vec \alpha)=\frac{1}{N^r}\#\left\{ \vec x\ne \vec y \in B(r,N): |\vec\alpha\cdot \vec a(\vec x)-\vec\alpha\cdot\vec a(\vec y)|\in \frac{I}{N^r}\right\},\numberthis\label{vecpaircor}\]
where $B(r,N)=\Z^r\cap [0,N]^r$. If $R_2(I,\vec a,N)(\vec \alpha)\to 2s$ as $N\to \infty$, the sequence $(\vec{\alpha}.\vec a(  {\vec n}))_{\vec n\in B(r,N)}$ has Poissonian pair correlation. The pair correlation of integer-valued and rational-valued vector sequences has been studied in  \cite{Boca} and \cite{chaubey16}, respectively. In the integer and rational cases, the authors showed that the PPC property could be obtained by counting solutions of Diophantine equations using harmonic analysis. For real vector sequences, we show that the PPC property is equivalent to bounds on the number of solutions of Diophantine inequalities written in \eqref{hyp} by slight modification in the proof of Theorem 2 of  \cite{Boca}. Since in real case, the situation is more delicate and obtaining \eqref{hyp} for sequences is difficult; therefore, we first reduce our problem to counting solutions of additive energy equations given in \eqref{addenergy} and \eqref{gammaenergy} with the help of a twisted moment of the Riemann zeta function. Our main results are as follows.
\begin{thm}\label{thm1}
let  $(\vec a(\vec x_n))_{n\ge 1}=(a^{1}(x_n^{1}),a^{2}(x_n^{2}),\cdots,a^{r}(x_n^{r}))_{n\ge 1}$ be a real-valued vector sequence such that for some $c_i>0$, $a^i(x_{n+1}^{i})-a^i(x_{n}^{i})\ge c_i$ for all $1\le i\le r$ and $n\ge 1$. Let $E_{N}^{r}$ be the number of solutions of the Diophantine inequality  
\[|j_1(\vec a(\vec x_{n_1})-\vec a(\vec x_{n_2}))-j_2(\vec a(\vec x_{n_3})-\vec a(\vec x_{n_4}))|<1,\numberthis\label{hyp}\] such that for sufficiently small $\epsilon>0$, we have $1\le j_1,j_2\le N^{r+\epsilon}$ and $\vec x_{n_1}\ne \vec x_{n_2},\vec x_{n_3}\ne \vec x_{n_4}\in B(r,N).$ If for some $\delta>0$  
\[E_{N}^{r}\ll N^{4r-\delta},\numberthis\label{hyp1}\]
the sequence   $\{\vec {\alpha}\cdot\vec a(\vec x_n)\}_{n\ge 1}$ has PPC for almost all $\vec{\alpha}\in \R^r$.
\end{thm}
The condition \eqref{hyp1} can be proved for certain sequences, such as the lacunary sequences. To get a wide range of examples, we eliminate the coefficients  $j_1$ and $j_2$ in \eqref{hyp} and obtain $r$ simultaneous conditions in the form of additive energy. 
\begin{thm}\label{thm2}
For $r\ge 2$, let $\{\vec a(\vec x_n)\}_{n\ge 1} =(a^{1}(x_n^{1}),a^{2}(x_n^{2}),\cdots,a^{r}(x_n^{r}))_{n\ge 1} $ be a real valued vector sequence such that for some $c_i>0$, \[a^i(x_{n+1}^{i})-a^i(x_{n}^{i})\ge c_i\numberthis\label{growthcon}\] for all $1\le i\le r$ and $n\ge 1$. For every $1\le i\le r$, let $E_{N}^{(i)*}$ be the number of solutions $(x^i_{n_1},x^i_{n_2},x^i_{n_3},x^i_{n_4})$ of the inequality 
\[|x^i_{n_1}-x^i_{n_2}+x^i_{n_3}-x^i_{n_4}|<1,\numberthis\label{addenergy}\]
where ${n_1},{n_2},{n_3},{n_4}\le N$. If there exists some $\delta>0$, such that for each  $1\le i\le r$, $E_{N}^{(i)*}\ll N^{\frac{280-136/r}{89}-\delta}$ as $N\to \infty $, then the sequence   $\{\vec {\alpha}\cdot\vec a(\vec x_n)\}_{n\ge 1}$ has PPC for almost all $\vec{\alpha}\in \R^r$.
\end{thm}

The above theorem yields further examples of real-valued vector sequences having PPC property. 
\begin{cor}\label{cor1}
Assume that $r\ge 2$ and for each $1\le i\le r$, let $\{\vec a(\vec x_n)\}_{n\ge 1} =(a^{1}(x_n^{1}),a^{2}(x_n^{2}),\cdots,a^{r}(x_n^{r}))_{n\ge 1} $ be a real-valued vector sequence such that  for each $1\le i\le r$, either $a^{i}(x_n^{i})$ is lacunary or a quadratic polynomial with real coefficients. Then for almost all $\vec{\alpha}\in \R^r$, the sequence$\{\vec {\alpha}\cdot\vec a(\vec x_n)\}_{n\ge 1}$ has PPC. Also, for $r\ge 3$, for each $1\le i\le r$, if $a^{i}(x_n^{i})$ is lacunary, a quadratic polynomial with real coefficients or a convex sequence, then $\{\vec {\alpha}\cdot\vec a(\vec x_n)\}_{n\ge 1}$ has PPC for almost all $\vec{\alpha}\in \R^r$.
\end{cor}
To prove Theorem \ref{thm2}, we show that the variance tends to zero, which we define later in section \ref{sec2}. We use the solutions of \eqref{addenergy} for the same. In Theorem \ref{thm2}, all solutions of \eqref{addenergy} do not contribute equally in estimating the variance. We can see this from the three different cases in section \ref{variance}. While instead of \eqref{addenergy}, if we consider \[|x^i_{n_1}-x^i_{n_2}+x^i_{n_3}-x^i_{n_4}|<\gamma \numberthis\label{gammacon}\]for some $\gamma\in (0,1]$ , then the smaller value of $\gamma$ gives a stronger effect on the variance. More precisely, here it is expected that number of solutions of \eqref{gammacon} is approximately equal to $\gamma$ times the number of solutions of \eqref{addenergy}. This is evident from the proof of Theorem \ref{thm2} in Section \ref{variance}. Hence, we replace condition \eqref{addenergy} with \eqref{gammacon} and obtain the following.
\begin{thm}\label{thm3}
For $r\ge 2$, let $(\vec a(\vec x_n))_{n\ge 1}=(a^{1}(x_n^{1}),a^{2}(x_n^{2}),\cdots,a^{r}(x_n^{r}))_{n\ge 1}$ be a real valued vector sequence such that for some $c_i>0$, $a^i(x_{n+1}^{i})-a^i(x_{n}^{i})\ge c_i$ for all $1\le i\le r$ and $n\ge 1$. Let $E_{N,\gamma}^{(i)*}$ be the number of solutions $(x^i_{n_1},x^i_{n_2},x^i_{n_3},x^i_{n_4})$ of the inequality \[|x^i_{n_1}-x^i_{n_2}+x^i_{n_3}-x^i_{n_4}|<\gamma .\numberthis\label{gammaenergy}\]
for each $1\le i\le r$ and ${n_1},{n_2},{n_3},{n_4}\le N$. If there exists some $\delta>0$ such that for each  $1\le i\le r$ and for all $\eta>0$, 
\[E_{N,\gamma}^{(i)*}\ll_{\eta,\delta} N^{2+\eta}+\gamma N^{3-\delta}\numberthis\label{thm3bound}\] as $N\to \infty $, then the sequence   $\{\vec {\alpha}\cdot\vec a(\vec x_n)\}_{n\ge 1}$ has PPC for almost all $\vec{\alpha}\in \R^r$.
\end{thm}
From \cite[Theorem 3]{aistletner21}, we know that for every real number $\theta_i>1$ and $n_i \in \N$, the sequence $\{{n_{i}}^{\theta_i}\}$ has $E_{N,\gamma}^{(i)*}\ll_{\eta,\delta} N^{2+\eta}+\gamma N^{3-\delta}$. Using this with Theorem \ref{thm3}, we deduce 
\begin{cor}\label{cor3}
   Let $r\ge 2$. For each $1\le i\le r$, for every real number $\theta_i>1$, and $n_i \in \N$, we consider $\vec a(\vec x_n) =({n_{1}}^{\theta_1},{n_{2}}^{\theta_2},\cdots,{n_{r}}^{\theta_r})$. Then for almost all $\vec{\alpha}\in \R^r$, the sequence$\{\vec {\alpha}\cdot\vec a(\vec x_n)\}_{n\ge 1}$ has PPC.
\end{cor}
\subsection{Organization}
This article is organized as follows: Section \ref{sec2} contains preliminary results required to prove Theorem \ref{thm1}, \ref{thm2}, and \ref{thm3}. Section \ref{expectations} covers estimates of the expected value. In Section \ref{variance}, we define the variance of the sequence and prove Lemma \ref{lemma1}, which gives a bound of variance.  In Section \ref{mainproof}, we prove Theorems \ref{thm1} and \ref{thm2}.  In section \ref{cor}, we prove Corollary \ref{cor1}. Finally, Section \ref{thm3proof} covers proof of Theorem \ref{thm3}.
\subsection{Notation}
\begin{itemize}
    \item We write $e(x)$ for $e^{2\pi \iota x},$ $\|.\|$ is used for the max norm,
   and $\iota$ denotes $\sqrt{-1}.$
    \item $\zeta(s)$ is the Riemann zeta function.
     \item  We use the Vinogradov $\ll$ asymptotic notation. Dependence on a parameter is denoted by a subscript.
     \item For a real number x, $\lfloor{x}\rfloor$ denotes the integer part of x and $\lceil{x}\rceil$ is equal to $\lfloor{x}\rfloor+1$.
    \item We use $\N, \Z, \R$ for sets of natural numbers, integers, and real numbers, respectively. 
  \item For simplicity, we take $\vec x$, $\vec y$, $\vec z$, and $\vec w$ instead of $\vec x_{n_1}, \vec x_{n_2}, \vec x_{n_3}$, and $\vec x_{n_4}$ in proof of Theorems. Also, we are taking $E_{N}^{*}=\max\{E_{N}^{(1)*},E_{N}^{(2)*},\cdots,E_{N}^{(r)*}\}$.
\end{itemize}
\section{Acknowledgment}
This work was supported by the University Grants Commission, Department of Higher Education, Government of India [DEC18-434199 to S.G.]; and the Science and Engineering Research Board, Department of Science and Technology, Government of India [SB/S2/RJN-053/2018 to S.C.].
\section{Preliminaries}\label{sec2}
Let $\I_{[-s/N^r,s/N^r]}(x)$ be the indicator function of interval $[-s/N^r,s/N^r]$ defined as
\[\I_{[-s/N^r,s/N^r]}(x)=\left\{\begin{array}{ll}
   1  &  \text{if} \ \{x\}\ \in \ [-s/N^r,s/N^r],\\
  0   & \text{otherwise}, 
\end{array}\right.\]
where $\{x\}$ is the fractional part of $x$. We intend to show that for almost all $\vec{\alpha}\in \R^r$, we have 
\[\frac{1}{N^r}\sum_{\vec x\ne \vec y \in B(r,N)}\I_{[-s/N^r,s/N^r]}(\vec\alpha\cdot \vec a(\vec x)-\vec\alpha\cdot\vec a(\vec y))\to 2s\numberthis\label{limitcon}\]
as $N\to \infty$ and for all $s\ge 0.$ Sometimes, instead of the above notation for the indicator function, we use the notation $\I(\|\vec\alpha\cdot \vec a(\vec x)-\vec\alpha\cdot\vec a(\vec y)\|\le s/N^r)$. For  further analysis, we replace $\I_{[-s/N^r,s/N^r]}(x)$ by Selberg polynomials $f^{\pm }_{K,s,N^r}$ (for details see \cite[Chapter 1]{Montgo}) of degree at most $K$ such that for all $x$, \[f^{- }_{K,s,N^r}(x)\le \I_{[-s/N^r,s/N^r]}(x)\le f^{+ }_{K,s,N^r}(x).\]
Moreover, the Selberg polynomials $f^{\pm }_{K,s,N^r}$ satisfy
\[\int_{0}^{1}f^{\pm }_{K,s,N^r}dx=\frac{2s}{N^r}\pm \frac{1}{K+1}.\numberthis\label{f0value}\]
If the Fourier series expansion of Selberg polynomials is given by $f^{\pm }_{K,s, N^r}(x)=\sum_{j\in \Z}c_j^{\pm}e^{2\pi ijx}$, then for all $j$, \[|c_j^{\pm}|\le \min\left(\frac{2s}{N^r},\frac{1}{\pi |j|}\right)+\frac{1}{K+1}.\numberthis\label{fcoeff}\]
For $|j|>K$, we have $c_j^{\pm}=0$. Thus, for every fixed integer $t$, and for all $s\ge 0$, \eqref{limitcon} is equivalent to the condition: for almost all $\vec{\alpha}\in \R^r$, as $N\to \infty$
\[\frac{1}{N^r}\sum_{\vec x\ne \vec y \in B(r,N)}f^{\pm }_{tN^r,s,N^r}(\vec\alpha\cdot \vec a(\vec x)-\vec\alpha\cdot\vec a(\vec y))\sim N^r\int_0^1f^{\pm }_{tN^r,s,N^r}(x)dx.\numberthis\label{limcon2}\]
 For a fixed  $t\in \N$ and a fixed positive real number $s$, we denote $f^{+}_{tN^r,s,N^r}$ (or $f^{-}_{tN^r,s,N^r}$) by $f_{N^r}.$  
 
 Since we are considering real-valued sequences, we do not have any periodicity. Therefore, to compute the expected value and the variance with respect to $\vec{\alpha}$ of the left-hand side of \eqref{limcon2},  we integrate over all $\vec{\alpha}\in \R^r$ with respect to an appropriate  measure $\mu$ defined as
\[d\mu(\vec x)=\dfrac{2sin^2(x_1/2)}{\pi x_1^2}\frac{2sin^2(x_2/2)}{\pi x_2^2}\cdots \frac{2sin^2(x_r/2)}{\pi x_r^2}dx_1dx_2\cdots dx_r,\]where $\vec{x}=(x_1,x_2,\cdots,x_r).$ The Fourier transform $x_i \to \dfrac{2sin^2(x_i/2)}{\pi x_i^2}$ for all $i\in \{1,2,\cdots,r\}$ is a non-negative real function, uniformly bounded, and the support of the function is $(-1,1)$. If the Fourier series expansion of $f_{N^r}(x)$ is $\sum_{j\in \Z}c_je(jx)$, then from \eqref{fcoeff}, we have   $c_j=0$ when $|j|>tN^r$, $|c_j|\le {1}/{N^r}$ for $0<|j|\le tN^r$, and $c_0=\int_0^1f_{N^r}(x)dx\ll 1/N^r.$ Next, we define 
\[\Phi(t):=e^{-t^2/2}.\numberthis\label{phidef}\]
The function $\Phi(t)$ has a positive Fourier transform $\hat{\Phi}(u)=\sqrt{2\pi}\Phi(u).$ Also, define the function $K$ as
\[K(t):=\frac{\sin^2({1+\epsilon/4})t\log N}{\pi t^2({1+\epsilon/4})\log N},\numberthis\label{kdef}\]with Fourier transform $\hat{K}(u)=\max\left(1-\dfrac{|u|}{2(1+\epsilon/4)\log N},0\right).$

In \cite{Tenen}, the authors derive a theorem for the average of a certain class of holomorphic function. We use this to estimate the complex integration in Lemma \ref{lemma5}.  
\begin{lem}\cite[Lemma 5.3]{Tenen}\label{averagelem}
   Let $\sigma \in (-\infty, 1)$ and $F$ be a holomorphic function in the strip $y=\Im{(z)}\in [\sigma-2,0]$, such that 
   \[\sup_{\sigma-2\le y\le 0}|F(x+\iota y)|\ll \frac{1}{x^2+1}.\numberthis\label{series}\] Then, for all $s=\sigma+\iota t\in \C, t\ne 0$, we have \[\sum_{k,l\ge 1}\frac{\hat{F}(\log kl)}{k^sl^{\Bar{s}}}=\int_{\R}\zeta(s+\iota u)\overline{\zeta(s-\iota u)}F(u)du+2\pi \zeta(1-2\iota t)F(\iota s-\iota)+2\pi\zeta(1+2\iota t)F(\iota \Bar{s}-\iota).\]
\end{lem}

Next, the expected value of the left side of \eqref{limcon2} with respect to $\vec {\alpha}$ is defined as
\[E(f_{N^r},\mu):=\int_{\R^r}\frac{1}{N^r}\sum_{\vec x\ne \vec y \in B(r,N)}f_{N^r}(\vec\alpha\cdot \vec a(\vec x)-\vec\alpha\cdot\vec a(\vec y))d\mu(\vec\alpha).\numberthis\label{expectedvalue}\]
 We also define the variance  as follows: if we write \[h_{N^r}(x)=f_{N^r}(x)-\int_0^1f(x)dx=\sum_{\substack{j\in \Z\\j\ne 0}}c_je(jx),\numberthis\label{hfunc}\]
    the variance is defined as 
    \[Var(h_{N^r},\mu):=\int_{\R^r}\left(\frac{1}{N^r}\sum_{\vec x\ne \vec y \in B(r,N)}h_{N^r}(\vec\alpha\cdot \vec a(\vec x)-\vec\alpha\cdot\vec a(\vec y))\right)^2d\mu(\vec\alpha).\numberthis\label{vardef}\]
To obtain \eqref{limcon2}, we first find the expectation and the variance, and then using the hypothesis of Theorems \ref{thm1}, \ref{thm2}, and \ref{thm3}, we show that the variance is not too large. After that, applying Chebyshev’s inequality and the first Borel–Cantelli lemma yields the desired results. See section \ref{mainproof}.
\section{Controlling Expectation}\label{expectations}
Let $f_{N^r}$ be the Selberg polynomial and $\mu$ be the Lebesgue measure defined in the previous section. Then, from \eqref{expectedvalue}, the expectation of the pair correlation function is given by  
\begin{align*}
 &  \left|\int_{\R^r}\frac{1}{N^r}\sum_{\vec x\ne \vec y \in B(r,N)}f_{N^r}(\vec\alpha\cdot \vec a(\vec x)-\vec\alpha\cdot\vec a(\vec y))d\mu(\vec\alpha)-N^r\int_0^1 f_{N^r}(x)dx\right|\\&=\left|\frac{1}{N^r}\sum_{\vec x\ne \vec y\in B(r,N)}\int_{\R^r}\sum_{j\in \Z}c_je(j(\vec\alpha\cdot \vec a(\vec x)-\vec\alpha\cdot\vec a(\vec y)))d\mu(\vec\alpha)-N^r\int_0^1 f_{N^r}(x)dx\right|\\
    &\le\left(N^r-\frac{N^r(N^r-1)}{N^r}\right)\int_0^1 f_{N^r}(x)dx+\frac{1}{N^r}\sum_{1\le |j|\le tN^r}|c_j|\sum_{\vec x\ne \vec y \in B(r,N)}\left|\int_{\R^r}e(j(\vec\alpha\cdot \vec a(\vec x)-\vec\alpha\cdot\vec a(\vec y)))d\mu(\vec\alpha)\right|\\
    &\ll \frac{1}{N^r}+\frac{1}{N^{2r}}\sum_{1\le |j|\le tN^r}\sum_{\vec x\ne \vec y \in B(r,N)}\I(\|j( \vec a(\vec x)-\vec a(\vec y))\|<1)\ll \frac{1}{N}.
    \end{align*}
    Here the last inequality is estimated using the growth condition of $a^i(x_i)$ for all $i\in \{1,2,\cdots,r\}$, and the fact $|c_j|\le 1/N^r.$ Hence, we can write 
    \[E(f_{N^r},\mu)=\int_{\R^r}\frac{1}{N^r}\sum_{\vec x\ne \vec y \in B(r,N)}f_{N^r}(\vec\alpha\cdot \vec a(\vec x)-\vec\alpha\cdot\vec a(\vec y))d\mu(\vec\alpha)=N^r\int_0^1 f_{N^r}(x)dx+\BigO{\frac{1}{N}}.\numberthis\label{expectedval}\]
    \section{Controlling Variance}\label{variance}
 We keep the setup as in Section \ref{sec2}. In this section, we study the variance. 
    First, we  square the inner part of the integral of the right side of \eqref{vardef}, and using the Fourier transform of the indicator function, we have 
   \begin{align*}
      Var(h_{N^r},\mu)&=\int_{\R^r} \frac{1}{N^{2r}}\sum_{\substack{\vec x, \vec y, \vec z,\vec w \in B(r,N)\\\vec x\ne\vec y, \vec z\ne \vec w}}\sum_{\substack{j_1,j_2\in \Z\\j_1,j_2\ne 0}}c_{j_1}c_{j_2}e(\vec \alpha \cdot(j_1(\vec a(\vec x)-\vec a(\vec y))-j_2(\vec a(\vec z)-\vec a(\vec w))))d\mu (\vec \alpha)\\
      &\ll \frac{1}{N^{4r}}\sum_{\substack{\vec x, \vec y, \vec z,\vec w \in B(r,N)\\\vec x\ne\vec y, \vec z\ne \vec w}}\sum_{1\le j_1,j_2\le tN^r}\I(\|j_1(\vec a(\vec x)-\vec a(\vec y))-j_2(\vec a(\vec z)-\vec a(\vec w))\|<1).\numberthis\label{varhyp1}
   \end{align*}
   We prove the following bound of  $Var(h_{N^r},\mu)$ in terms of the additive energy $E_{N}^{*}$.
    \begin{lem}\label{lemma1}
       For $r\ge 2$, $\epsilon>0$, and $E_{N}^{*}=\max\{E_{N}^{(1)*},E_{N}^{(2)*},\cdots,E_{N}^{(r)*}\}$, we have
       \begin{align*}
      Var(h_{N^r},\mu)\ll \max\left(N^{\frac{58-241r}{89}+\epsilon}(E_N^{*})^{\frac{76r}{89}},N^{\frac{59-280r+\epsilon}{89}}(E_{N}^{*})^{r},N^{\frac{116+39k-280r}{89}}(E_{N}^*)^{\frac{89r-13k}{89}}\right), 
 \end{align*}
    as $N\to \infty.$ Here, $k\in \N$ such that $1\le k\le r.$
    \end{lem}
    \begin{proof}[Proof of Lemma \ref{lemma1}]
    To prove Lemma \ref{lemma1},  we begin with \eqref{vardef}. First, we break the sum over $j_1$ and $j_2$ into dyadic intervals and use Cauchy-Schwarz inequality. Let $U$ be the smallest integer such that $2^U\ge tN^r$, we have
   \begin{align*}
Var(h_{N^r},\mu)&=\int_{\R^r}\left(\frac{1}{N^r}\sum_{u=1}^U\sum_{\vec x\ne \vec y \in B(r,N)}\sum_{2^{u-1}\le |j|<2^u }c_je(j(\vec\alpha\cdot \vec a(\vec x)-\vec\alpha\cdot\vec a(\vec y)))\right)^2d\mu(\vec\alpha)\\&\ll \frac{1}{N^{2r}}\int_{\R^r}\left(\sum_{k=1}^{U}1\right)\sum_{u=1}^{U}\left|\sum_{\vec x\ne \vec y \in B(r,N)}\sum_{2^{u-1}\le |j|<2^u }c_je(j(\vec\alpha\cdot \vec a(\vec x)-\vec\alpha\cdot\vec a(\vec y)))\right|^2d\mu(\vec\alpha)\\
       &\ll \frac{\log N}{N^{4r}}\sum_{u=1}^{U}\sum_{\substack{\vec x, \vec y, \vec z,\vec w \in B(r,N)\\\vec x\ne\vec y, \vec z\ne \vec w}}\sum_{2^{u-1}\le j_1,j_2<2^u}\I(\|j_1(\vec a(\vec x)-\vec a(\vec y))-j_2(\vec a(\vec z)-\vec a(\vec w))\|<1).\numberthis\label{varadd}
       \end{align*}
 Let $\vec \rho (\vec u)=(\rho_1(u_1), \rho_2(u_2),\cdots,\rho_r(u_r))$ such that $\rho_i(u_i)=a^i(x_i)-a^i(y_i)$ for all $1\le i\le r$ and  similarly define $\vec \rho (\vec v)=(\rho_1(v_1), \rho_2(v_2),\cdots,\rho_r(v_r))$ such that $\rho_i(v_i)=a^i(z_i)-a^i(w_i)$ for all $1\le i\le r$. So, there are $M=N^2-N$ possible values of $u_i$ and $v_i$ for all $1\le i\le r.$ Now, for a fixed $u\le 2tN^r$, estimation of the above sum is equivalent to calculating the sum
 \[\prod_{i=1}^r\sum_{\substack{1\le u_i,v_i\le M\\ \text{for some  i}\ u_i\ne v_i}}\sum_{2^{u-1}\le j_1,j_2<2^u}\I(|j_1\rho_i(u_i)-j_2\rho_i(v_i)|<1).\numberthis\label{varsum}\]
  Next, we split the sum into three cases depending on the size of $\min\{\rho_i(u_i),\rho_i(v_i)\}.$ In case 1, we consider $\min\{\rho_i(u_i),\rho_i(v_i)\}>N^{1.01}$ for all $1\le i\le r$. In case 2, we assume $\min\{\rho_i(u_i),\rho_i(v_i)\}\le N^{1.01}$ for all $1\le i\le r$, and in the last case, we assume that for some $k\le r$, $\min\{\rho_i(u_i),\rho_i(v_i)\}>N^{1.01}$ for $1\le i\le k$, and $\min\{\rho_i(u_i),\rho_i(v_i)\}\le N^{1.01}$ for $k+1\le i\le r$.
     \subsection{Estimation of the number of solutions in case 1 }  
    Let $\min\{\rho_i(u_i),\rho_i(v_i)\}>N^{1.01}$ for all $1\le i\le r$. For integers $k_1,k_2,\cdots,k_r\ge N^{1.01}$, we define 
   \[b(k_i):=\sum_{u_i=1}^{M}\I(\rho_i(u_i)\in[k_i,k_i+1)).\numberthis\label{bkdef}\] This implies 
    \[\sum_{k_i=1}^{\infty}b(k_i)\le M.\]
    Set $T=2^{u/r}N^{1+\epsilon/2}$ and choose $u$  such that $2^{u-1}\le j_1,j_2<2^u$. Now, we split the interval $[1,\infty)$ into disjoint union of interval $I_{h_i}$ for $h_i\ge 0$ defined as
    \[I_{h_i}:=\left[\left\lceil\left(1+\frac{1}{T}\right)^{h_i}\right\rceil,\left\lceil\left(1+\frac{1}{T}\right)^{h_i+1}\right\rceil\right].\] Using $b({k_i})$, for $h_i\ge 0$, we set 
    \[a(h_i):=\left(\sum_{k_i\in I_{h_i}}(b(k_i))^2\right)^{1/2}.\numberthis\label{ahdef}\]
To obtain the number of solutions of \eqref{varsum} in case 1, we prove the following Lemmas. First, we convert the counting problem in \eqref{varsum} in terms of $a(h_i)$ for $1\le i\le r$.
  \begin{lem}\label{lemma3}
    We have 
\begin{align*}
    \sum_{2^{u-1}\le j_1,j_2<2^u}\prod_{i=1}^r\sum_{1\le u_i,v_i\le M}\I(|j_1\rho_i(u_i)-j_2\rho_i(v_i)|<1) \ll  \sum_{2^{u-1}\le j_1,j_2<2^u} \prod_{i=1}^r\sum_{\substack{h_i,g_i\ge 0\\ \left|(1+\frac{1}{T})^{h_i-g_i}-\frac{j_2}{j_1}\right|\le \frac{4}{T}}}a(h_i)a(g_i).
\end{align*}
  \end{lem}  
  \begin{proof}
    Let $j_1$ and $j_2$ be fixed and $j_1\ge j_2$. Assume that for $1\le i\le r$, $k_i\ge N^{1.01}$ is an integer in $I_{h_i}$, and let $\rho_i(u_i)\in[k_i,k_i+1).$ Then for $1\le i\le r$, the inequality $|j_1\rho_i(u_i)-j_2\rho_i(v_i)|<1$ is only possible if
    \[\left|\left\lceil\frac{j_1k_i}{j_2}\right\rceil- \rho_i(v_i)\right|<4.\numberthis\label{eqncond}\]
    Next, define a mapping $k_i\to \eta(k_i)$ where $\eta(k_i)=\left\lceil\frac{j_1k_i}{j_2}\right\rceil.$ This gives
    \begin{align*}
        \sum_{\substack{\rho_i(u_i)\in I_{h_i}\\\rho_i(v_i)\in I_{g_i} }}\I(|j_1\rho_i(u_i)-j_2\rho_i(v_i)|<1)&\ll \sum_{k_i\in I_{h_i}}\sum_{\rho_i(u_i)\in [k_i,k_i+1)}\sum_{\substack{\rho_i(v_i)\in I_{g_i}\\|\eta(k_i)-\rho_i(v_i)|<4}}1 
        \\& \ll  \sum_{k_i\in I_{h_i}}\sum_{\rho_i(u_i)\in [k_i,k_i+1)}\sum_{-4\le v\le 3}\sum_{\substack{\rho_i(v_i)\in I_{g_i}\\ \rho_i(v_i)\in [\eta(k_i)+v,\eta(k_i)+v+1)}}1
        \\&\ll \sum_{-4\le v\le 3}\sum_{\substack{\rho_i(u_i)\in I_{h_i}\\ \eta(k_i)+v\in I_{g_i}}}b(k_i)b(\eta(k_i)+v)\\&
        \ll a(h_i)a(g_i). \numberthis\label{inequality2}
    \end{align*}
    Next, for fixed $j_1$ and $j_2$, such that  $j_1\ge j_2$ and  $1\le i\le r$, assume that $\rho_i(u_i)\in I_{h_i},\rho_i(v_i)\in I_{g_i}$ such that $|j_1\rho_i(u_i)-j_2\rho_i(v_i)|<1.$ This implies 
    \begin{align*}
        \left|\frac{\rho_i(u_i)}{\rho_i(v_i)}-\frac{j_2}{j_1}\right|< \frac{1}{j_1\rho_i(v_i)}\le \frac{1}{2^{u-1}N^{1.01}}\le \frac{1}{T}.\numberthis\label{inequality3}
    \end{align*}
    Since $\rho_i(u_i)\in I_{h_i},\rho_i(v_i)\in I_{g_i}$, therefore $\dfrac{\rho_i(u_i)}{\rho_i(v_i)}$ lies between $(1+1/T)^{h_i-g_i-1}$ and $(1+1/T)^{h_i-g_i+1}$, and so we have
    \[\left|\frac{\rho_i(u_i)}{\rho_i(v_i)}-\left(1+\frac{1}{T}\right)^{h_i-g_i}\right|\le \frac{3}{T}.\numberthis\label{ineqaulity4}\]
    From \eqref{inequality3} and \eqref{ineqaulity4}, we have 
    \[\left|\left(1+\frac{1}{T}\right)^{h_i-g_i}-\frac{j_2}{j_1}\right|\le \frac{4}{T}.\numberthis\label{ineqaulity5}\]
    For fixed $j_1$ and $j_2$, and for  $1\le i\le r$, \eqref{inequality2} and \eqref{ineqaulity5} give 
    \[\sum_{1\le u_i,v_i\le M}\I(|j_1\rho_i(u_i)-j_2\rho_i(v_i)|<1) \ll  \sum_{\substack{h_i,g_i\ge 0\\ \left|(1+\frac{1}{T})^{h_i-g_i}-\frac{j_2}{j_1}\right|\le \frac{4}{T}}}a(h_i)a(g_i).\]
 We obtain the required result of Lemma \ref{lemma3} by taking the product over $i$ and sum over $j_1$ and $j_2$.
   \end{proof} 
Next, we write the right side of Lemma \ref{lemma3} in terms of complex integration. 

 \begin{lem}\label{lemma4}
 For the function $K(t)$ defined in \eqref{kdef}, we have
 \begin{align*}
  &\sum_{2^{u-1}\le j_1,j_2<2^u} \prod_{i=1}^r\sum_{\substack{h_i,g_i\ge 0\\ \left|(1+\frac{1}{T})^{h_i-g_i}-\frac{j_2}{j_1}\right|\le \frac{4}{T}}}a(h_i)a(g_i)\\&\ll \frac{2^{u}}{T^r}\int_{\R^r}\prod_{i=1}^r \sum_{j_1,j_2\ge 1}\frac{\hat{K}(\log j_1j_2)}{(j_1j_2)^{1/2}}\left(\frac{j_1}{j_2}\right)^{\iota s_i}|P(s_i)|^2\phi\left(\frac{s_i}{T}\right)ds_i,   
 \end{align*}
 where the function $P(s_i)$ is defined as\[P(s_i):=\sum_{h_i\ge 0}a(h_i)\left(1+\frac{1}{T}\right)^{\iota h_is_i}. \numberthis\label{psdef}\]
 \end{lem} 
 \begin{proof}
   From the definitions and non negative property of $\hat{K}$ and $\Phi$, we have \begin{align*}
       &\sum_{2^{u-1}\le j_1,j_2<2^u} \frac{1}{(j_1j_2)^{1/2}}\prod_{i=1}^r\sum_{\substack{h_i,g_i\ge 0\\\left|(1+\frac{1}{T})^{h_i-g_i}-\frac{j_2}{j_1}\right|\le \frac{4}{T}}}a(h_i)a(g_i)\\
       \ll & \sum_{j_1,j_2\ge 1}\frac{\hat{K}(\log j_1j_2)}{(j_1j_2)^{1/2}}\prod_{i=1}^r\sum_{h_i,g_i\ge 0}a(h_i)a(g_i)\hat{\Phi}\left(T\log\left(\frac{j_1}{j_2}(1+1/T)^{h_i-g_i}\right)\right)\\
       &\ll \frac{1}{T^r}\prod_{i=1}^r\int_{\R}\sum_{j_1,j_2\ge 1}\frac{\hat{K}(\log j_1j_2)}{(j_1j_2)^{1/2}}\left(\frac{j_1}{j_2}\right)^{\iota s_i}|P(s_i)|^2\phi\left(\frac{s_i}{T}\right)ds_i.\numberthis\label{prodsum}
   \end{align*}
   Next, using the fact $2^u\ll (j_{1}j_{2})^{1/2}\ll 2^u$, we obtain the required result.
 \end{proof}
 To solve the complex integration in Lemma \ref{lemma4}, we require the following Lemma. 
\begin{lem}\label{lemma2}
    For $\Phi(s)$ define in $\eqref{phidef}$, we have 
    \[\prod_{i=1}^{r}\int_{\R}|P(s_i)|^2\Phi(s_i/T)ds_i\ll T^r({E_N^*})^r.\]
\end{lem}
\begin{proof}
   From \eqref{phidef} and \eqref{psdef}, we have
   \begin{align*}
     \prod_{i=1}^{r} \int_{\R}|P(s_i)|^2\Phi(s_i/T)ds_i&= \prod_{i=1}^{r}\int_{\R}\sum_{h_i,g_i\ge 0}a(h_i)a(g_i)\left(1+\frac{1}{T}\right)^{\iota(h_i-g_i)s_i}\Phi(s_i/T)ds_i\\
      &= T^r\prod_{i=1}^{r}\sum_{h_i,g_i\ge 0}a(h_i)a(g_i)\int_{\R}exp\left(\iota T(h_i-g_i)s_i\log\left(1+\frac{1}{T}\right)\right)\Phi(s_i)ds_i\\
      &=T^r\prod_{i=1}^{r}\sum_{h_i,g_i\ge 0}a(h_i)a(g_i)\hat{\Phi}\left( T(h_i-g_i)\log\left(1+\frac{1}{T}\right)\right)\\&
      \ll T^r\prod_{i=1}^{r}\sum_{h_i,g_i\ge 0}a(h_i)a(g_i)\hat{\Phi}\left( \frac{h_i-g_i}{2}\right)\ll T^r\prod_{i=1}^{r}\sum_{h_i\ge 0}(a(h_i))^2 \\ &\ll  T^r\prod_{i=1}^{r}\sum_{k_i=1}^{\infty}(b(k_i))^2\ll  T^r\prod_{i=1}^{r} \sum_{\substack{u_i,v_i=1\\|\rho_i(u_i)-\rho_i(v_i)|<1}}^{M}1 \le T^r({E_N^*})^r.
   \end{align*}
\end{proof}
 Finally, in the next Lemma, we solve the complex integration written on the right side of Lemma \ref{lemma4} using Lemmas \ref{lemma2} and \ref{averagelem}.
 \begin{lem}\label{lemma5}
 We have
    \[\frac{2^{u}}{T^r}\int_{\R^r}\prod_{i=1}^r\sum_{j_1,j_2\ge 1}\frac{\hat{K}(\log j_1j_2)}{(j_1j_2)^{1/2}}\left(\frac{j_1}{j_2}\right)^{\iota s_i}|P(s_i)|^2\phi\left(\frac{s_i}{T}\right)ds_i \ll N^{\frac{58+115r}{89}+\epsilon}(E_N^{*})^{76r/89}. \]
 \end{lem}   
    \begin{proof}
      From \eqref{kdef}, it is easy to see that the function $K(x+\i t)$ satisfies the assumption \eqref{series} for the required range of $y.$ From Lemma \ref{averagelem}, we have 
      \[\int_{\R^r}\prod_{i=1}^r \sum_{j_1,j_2\ge 1}\frac{\hat{K}(\log j_1j_2)}{(j_1j_2)^{1/2}}\left(\frac{j_1}{j_2}\right)^{\iota s_i}|P(s_i)|^2\phi\left(\frac{s_i}{T}\right)ds_i=I_1+I_2+I_3,\numberthis\label{rproduct}\]
      where
      \begin{align*}
          I_1&=\int_{\R^r} \prod_{i=1}^r |P(s_i)|^2\phi\left(\frac{s_i}{T}\right)\left(\int_{\R}\zeta(1/2+\iota\sum_{i=1}^{r}s_i+\iota u)\zeta(1/2-\iota\sum_{i=1}^{r}s_i+\iota u)K(u)du\right)ds_i\\
          &=\int_{\R^r} \prod_{i=1}^r |P(s_i)|^2\phi\left(\frac{s_i}{T}\right)\left(\int_{|u|\ge T^r}\zeta(1/2+\iota\sum_{i=1}^{r}s_i+\iota u)\zeta(1/2-\iota\sum_{i=1}^{r}s_i+\iota u)K(u)du\right)ds_i\\&+\int_{\R^r} \prod_{i=1}^r |P(s_i)|^2\phi\left(\frac{s_i}{T}\right)\left(\int_{|u|\le T^r}\zeta(1/2+\iota\sum_{i=1}^{r}s_i+\iota u)\zeta(1/2-\iota\sum_{i=1}^{r}s_i+\iota u)K(u)du\right)ds_i\\&=I_{11}+I_{12},
      \end{align*}
     \begin{align*}
          I_2&= 2\pi\int_{\R^r}\prod_{i=1}^r \zeta\left(1-2\iota\sum_{i=1}^{r}s_i\right)K\left(-\sum_{i=1}^{r}s_i-\iota/2\right)|P(s_i)|^2\phi\left(\frac{s_i}{T}\right)ds_i,
      \end{align*}
      and \begin{align*}
          I_3&= 2\pi\int_{\R^r}\prod_{i=1}^r \zeta\left(1+2\iota\sum_{i=1}^{r}s_i\right)K\left(\sum_{i=1}^{r}s_i-\iota/2\right)|P(s_i)|^2\phi\left(\frac{s_i}{T}\right)ds_i.    \end{align*}
         From \cite[Chapter II.3]{Tenenbaum}, we have $\zeta(1+\iota t)\ll \log t$, for all $1\le i\le r$ \[|P(s_i)|^2\le|P(0)|^2=\left(\sum_{h_i\ge 0}a(h_i)\right)^2\ll \left(\sum_{k_i\ge 1}b(k_i)\right)^2\le N^4,\]
          and also,
          \[\left|K\left(\pm\sum_{i=1}^{r}s_i-\iota/2\right)\right|\ll \frac{N^{1+\epsilon/4}}{(\sum_{i=1}^{r}s_i+1)^2}.\]
          Using the above results, we have
          \begin{align*}
              |I_2|,|I_3|&\ll N^{4r}N^{1+\epsilon/4}\int_{\R^r} \prod_{i=1}^r \frac{\log\left(\sum_{i=1}^r s_i\right)}{(\sum_{i=1}^{r}s_i+1)^2}ds_i\ll N^{4r+1+\epsilon/4}.\numberthis\label{i2}
          \end{align*}
         Next, we estimate the value of $I_{11}$ using the bound $\zeta(1/2+\iota t)\ll |t|^{1/6}$ (see \cite{Tenenbaum}) and $K(u)\ll u^{-2}.$
         \begin{align*}
            I_{11}&\ll \int_{\R^r} \prod_{i=1}^r |P(s_i)|^2\phi\left(\frac{s_i}{T}\right)\left(\int_{|u|\ge T^r}\left(\sum_{i=1}^{r}|s_i|^{1/3}+|u|^{1/3}\right)K(u)du\right)ds_i\\
            & = \int_{\R^r}\prod_{i=1}^r \left(\sum_{i=1}^{r}|s_i|^{1/3}\right)|P(s_i)|^2\phi\left(\frac{s_i}{T}\right)\left(\int_{|u|\ge T^r}K(u)du\right)ds_i\\&+\int_{\R^r} \prod_{i=1}^r |P(s_i)|^2\phi\left(\frac{s_i}{T}\right)\left(\int_{|u|\ge T^r}|u|^{1/3}K(u)du\right)ds_i\\&
            \ll \int_{\R^r}\prod_{i=1}^r \frac{\left(\sum_{i=1}^{r}|s_i|^{1/3}\right)}{T^{r}}|P(s_i)|^2\phi\left(\frac{s_i}{T}\right)ds_i\\&+\frac{1}{T^{2r/3}}\int_{\R^r} \prod_{i=1}^r |P(s_i)|^2\phi\left(\frac{s_i}{T}\right)\left(\int_{|u|\ge T^r}|u|^{1/3}K(u)du\right)ds_i.
         \end{align*}
         Now, from Lemma \ref{lemma2}, we have 
         \[I_{11}\ll T^{r/3}(E_{N}^{*})^r.\numberthis\label{i11}\]
         For $A=\frac{178}{13}$, by Ivi\'c's theorem \cite[Theorem 8.3]{ivic}, we
         have
         \[\int_0^{T}|\zeta(1/2+\iota t)|^Adt\ll T^{2+\frac{3(A-12)}{22}+\epsilon}=T^{29/13+\epsilon}.\numberthis\label{ivicthm}\]
         One can write \[|I_{12}|\le \int_{|u|\le T^r}K(u)\left(\int_{\R^r} \prod_{i=1}^r|\zeta(1/2+\iota\sum_{i=1}^{r}s_i+\iota u)||\zeta(1/2-\iota\sum_{i=1}^{r}s_i+\iota u)||P(s_i)|^2\phi\left(\frac{s_i}{T}\right)ds_i\right)du.\numberthis\label{i12}\]
         Using H\"older's inequality with parameters $1/A+1/A+1/B=1$ and \eqref{ivicthm}, the innermost integral of \eqref{i12} becomes 
         \begin{align*}
             &\int_{\R^r} \prod_{i=1}^r\zeta(1/2+\iota\sum_{i=1}^{r}s_i+\iota u)\zeta(1/2-\iota\sum_{i=1}^{r}s_i+\iota u)|P(s_i)|^2\phi\left(\frac{s_i}{T}\right)ds_i\\
             &\ll \int_{\R^r} \left(\int_{\R^r}\prod_{i=1}^r|\zeta(1/2+\iota\sum_{i=1}^{r}s_i+\iota u)|^A\phi\left(\frac{s_i}{T}\right)ds_i\right)^{1/A}\left(\int_{\R^r}\prod_{i=1}^r|\zeta(1/2-\iota\sum_{i=1}^{r}s_i+\iota u)|^A\phi\left(\frac{s_i}{T}\right)ds_i\right)^{1/A}\\& \times |P(0)|^{2r(1-1/B)}\left(\int_{\R^r} \prod_{i=1}^r|P(s_i)|^2\phi\left(\frac{s_i}{T}\right)ds_i\right)^{1/B}\\
             &\ll (T^{29/13+\epsilon})^{2/A}N^{4r(1-1/B)}T^{r/B}(E_{N}^{*})^{r/B}=(T^{29/13+\epsilon})^{2/A}N^{8r/A}T^{r(A-2)/A}(E_{N}^{*})^{r(A-2)/A}.
         \end{align*}
         Integrating with respect to $u$, we obtain 
         \begin{align*}
              I_{12}&\ll(T^{29/13+\epsilon})^{2/A}N^{8r/A}T^{r(A-2)/A}(E_{N}^{*})^{r(A-2)/A}\\&=T^{(29+76r)/89+\epsilon}N^{52r/89}(E_{N}^{*})^{76r/89}.\numberthis\label{i12final} \end{align*}
              From \eqref{rproduct}, \eqref{i2}, \eqref{i11}, and \eqref{i12final}, we have
              \begin{align*}
                &\frac{2^{u}}{T^r}  \int_{\R^r}\prod_{i=1}^r \sum_{j_1,j_2\ge 1}\frac{\hat{K}(\log j_1j_2)}{(j_1j_2)^{1/2}}(\frac{j_1}{j_2})^{\iota s_i}|P(s_i)|^2\phi\left(\frac{s_i}{T}\right)ds_i\\&\ll \frac{2^{u}}{T^r} (N^{4r+1+\epsilon/4}+T^{r/3}(E_{N}^{*})^r+T^{(29+76r)/89+\epsilon}N^{52r/89}(E_{N}^{*})^{76r/89}).
              \end{align*}
              Taking $T=2^{u/r}N^{1+\epsilon/2}$ and using $2^u\ll N^r$, we obtain the required result.
    \end{proof}
    \begin{rem}
  Combination of  Lemmas \ref{lemma3}, \ref{lemma4}, and \ref{lemma5} yields the result for Case 1.
    \[\sum_{2^{u-1}\le j_1,j_2<2^u}\prod_{i=1}^r\sum_{1\le u_i,v_i\le M}\I(|j_1\rho_i(u_i)-j_2\rho_i(v_i)|<1)\ll N^{\frac{58+115r}{89}+\epsilon}(E_N^{*})^{76r/89}.\numberthis\label{variance1} \]
    \end{rem}
     \subsection{Estimation of the number of solutions in case 2} First, suppose that $\max\{\rho_i(u_i),\rho_i(v_i)\}\le 4N^{1/4}$ for all $1\le i\le r$. Since we assume that $a^i(x_i+1)-a^i(x_i)\ge c>0$ that is $\rho_i(u_i)>c$ for all $1\le i\le r$, therefore for a fixed $i$, at most $N^{5/4}$ $\rho_i(u_i)$'s  can be smaller than $4N^{1/4}$. For fixed $j_2,\rho_i(u_i)$ and $\rho_i(v_i)$, there are at most $\ll 1$ possible choices for $j_1$
    such that $|j_1\rho_i(u_i)-j_2\rho_i(v_i))|<1.$ Hence, for $\max\{\rho_i(u_i),\rho_i(v_i)\}\le 4N^{1/4}$, we have
    \[\sum_{2^{u-1}\le j_1,j_2<2^u}\prod_{i=1}^r\sum_{1\le u_i,v_i\le M}\I(|j_1\rho_i(u_i)-j_2\rho_i(v_i)|<1)\ll 2^uN^{5r/2}\ll N^{7r/2}. \numberthis\label{case21}\]
    
    Next, we consider $\max\{\rho_i(u_i),\rho_i(v_i)\}\ge 4N^{1/4}$ for all $1\le i\le r$. One can see that $|j_1\rho_i(u_i)-j_2\rho_i(v_i)|<1$ is equivalent to $\left|\frac{j_1}{j_2}-\frac{\rho_i(v_i)}{\rho_i(u_i)}\right|<\frac{1}{j_2\rho_i(u_i)}.$ By triangle inequality, we have
    \[\frac{\rho_i(v_i)}{\rho_i(u_i)}\le \frac{j_1}{j_2}+\frac{1}{j_2}\le 4.\]
    Similarly, we have $\frac{\rho_i(v_i)}{\rho_i(u_i)}\ge \frac{1}{4}$, and so $\frac{\rho_i(v_i)}{\rho_i(u_i)}\in [1/4,4].$ Thus, it is sufficient to consider approximately $\log N$ intervals of the form $[4N^{\beta},8N^{\beta})$ where $\beta \in [1/4,1.01]$ such that $\rho_i(v_i)\in [4N^{\beta},8N^{\beta}).$ It implies $\rho_i(u_i)\in [N^{\beta},32N^{\beta})$ and \[\sum_{2^{u-1}\le j_1,j_2<2^u}\prod_{i=1}^r\sum_{\substack{1\le u_i,v_i\le M\\\rho_i(v_i), \rho_i(u_i)\in [N^{\beta},32N^{\beta}) }}\I(|j_1\rho_i(u_i)-j_2\rho_i(v_i)|<1)\numberthis\label{case2count}\] counts more solutions. We solve this case similarly as case 1 with the restriction $\rho_i(v_i), \rho_i(u_i)\in [N^{\beta},32N^{\beta}).$ Therefore, we define 
     \[b_2(k_i):=\sum_{\substack{1\le u_i\le M\\\rho_i(u_i)\in [N^{\beta},32N^{\beta}) }}\I(\rho_i(u_i)\in[k_i,k_i+1)).\]
 Note that by Cauchy-Schwarz inequality, we have
$\sum_{k_i=1}^{\infty}b_2(k_i)\ll N^{\beta/2}\sqrt{E_N^*}.$ Set $T=2^{u/r}N^{\beta}$, and corresponding to $b_2(k_i) $, we define $a(h_i)$, $P(s_i)$ similarly as in \eqref{ahdef} and \eqref{psdef}, respectively. Also, in this case $|P(0)|^2\ll E_N^*N^{\beta}$. Following case 1, we obtain integrals $I_1$, $I_2$ and $I_3$ as in \eqref{rproduct}, and which are estimated as
    \[I_2,I_3\ll N^{1+\beta r+\epsilon/4}(E_N^*)^r,\]and 
    \[I_1\ll T^{(29+76r)/89+\epsilon}N^{13\beta r/89}(E_{N}^{*})^{r}. \]
    Substitute $T=2^{u/r}N^{\beta}$ and $\beta\le 1.01$, this gives
    \begin{align*}
        &\sum_{2^{u-1}\le j_1,j_2<2^u}\prod_{i=1}^r\sum_{\substack{1\le u_i,v_i\le M\\\rho_i(v_i), \rho_i(u_i)\in [N^{\beta},32N^{\beta}) }}\I(|j_1\rho_i(u_i)-j_2\rho_i(v_i)|<1)\\
        &\ll \frac{2^u}{T^r}(N^{1+\beta r+\epsilon/4}(E_N^*)^r+T^{(29+76r)/89+\epsilon}N^{13\beta r/89}(E_{N}^{*})^{r})\\& \ll N^{1+0.1r+\epsilon/4}(E_N^*)^r+N^{\frac{58.29+76r+\epsilon}{89}}(E_{N}^{*})^{r}.\numberthis\label{variance2}
    \end{align*}
    \subsection{Estimation of the number of solutions in case 3} For some $k\ge 1$ and  $1\le i\le k$,  assume that  $\min\{\rho_i(u_i),\rho_i(v_i)\}>N^{1.01}$, and $\min\{\rho_i(u_i),\rho_i(v_i)\}\le N^{1.01}$ for $k+1\le i\le r$ (here we are rearranging the terms according to the value of $\min\{\rho_i(u_i),\rho_i(v_i)\}$). We can write \begin{align*}
&\prod_{i=1}^r\sum_{\substack{1\le u_i,v_i\le M\\ \text{for some  $i$}\ u_i\ne v_i}}\sum_{2^{u-1}\le j_1,j_2<2^u}\I(|j_1\rho_i(u_i)-j_2\rho_i(v_i)|<1)\\&\le \prod_{i=1}^k\sum_{\substack{1\le u_i,v_i\le M\\ \text{for some  $i$}\ u_i\ne v_i}}\sum_{2^{u-1}\le j_1,j_2<2^u}\I(|j_1\rho_i(u_i)-j_2\rho_i(v_i)|<1)\\&\times\prod_{l=1}^{r-k}\sum_{\substack{1\le u_l,v_l\le M\\ \text{for some  $l$}\ u_l\ne v_l}}\sum_{2^{u-1}\le j_1,j_2<2^u}\I(|j_1\rho_l(u_l)-j_2\rho_l(v_l)|<1).
   \end{align*}
    We follow case 1 for the first sum of the above equation, and for the second sum, we follow case 2. This yields
     \begin{align*}
\prod_{i=1}^r\sum_{\substack{1\le u_i,v_i\le M\\ \text{for some  i}\ u_i\ne v_i}}\sum_{2^{u-1}\le j_1,j_2<2^u}\I(|j_1\rho_i(u_i)-j_2\rho_i(v_i)|<1)\ll N^{\frac{116+39k-280r}{89}}(E_{N}^*)^{\frac{89r-13k}{89}}.\numberthis\label{variance3}
     \end{align*}
     On substituting values from \eqref{variance1}, \eqref{variance2} and \eqref{variance3} in \eqref{varsum}, we have 
 \begin{align*}
   &\prod_{i=1}^r\sum_{\substack{1\le u_i,v_i\le M\\ \text{for some  i}\ u_i\ne v_i}}\sum_{2^{u-1}\le j_1,j_2<2^u}\I(|j_1\rho_i(u_i)-j_2\rho_i(v_i)|<1) \\
   & \ll N^{\frac{58+115r}{89}+\epsilon}(E_N^{*})^{76r/89}+N^{\frac{29+76r+29\beta+(1+\beta)\epsilon}{89}}(E_{N}^{*})^{r}+N^{\frac{116+39k+76r}{89}}(E_{N}^*)^{\frac{89r-13k}{89}}. \numberthis\label{finalvar}
 \end{align*} 
 Next, from \eqref{varadd} and \eqref{finalvar},  we have
 \begin{align*}
      Var(h_{N^r},\mu)\ll \max\left(N^{\frac{58-241r}{89}+\epsilon}(E_N^{*})^{76r/89},N^{\frac{29-280r+29\beta+(1+\beta)\epsilon}{89}}(E_{N}^{*})^{r},N^{\frac{116+39k-280r}{89}}(E_{N}^*)^{\frac{89r-13k}{89}}\right).
 \end{align*}
  This proves Lemma 
  \ref{lemma1}.  
  \end{proof}
        \section{Proof of Theorems \ref{thm1} and \ref{thm2} }\label{mainproof}
  To prove Theorems \ref{thm1} and \ref{thm2}, we show that the variance is small. In Theorem \ref{thm1}, we assume that  $E_{N}^{r}\ll N^{4r-\delta}.$ On substituting this estimate in  \eqref{varhyp1}, we have $ Var(h_{N^r},\mu)\ll N^{-\delta}$. Similarly, for Theorem \ref{thm2}, using the hypothesis that $E_{N}^{(i)*}\ll N^{\frac{280-136/r}{89}-\delta}$ as $N\to \infty $ in Lemma \ref{lemma1} gives $ Var(h_{N^r},\mu)\ll N^{-\delta'}$ for sufficiently small ${\delta'}$ for Theorem \ref{thm2}. 
    Finally, invoking Chebyshev’s inequality, the first Borel–Cantelli lemma, the expected value \eqref{expectedval}, and variance of the left side of \eqref{limcon2} obtained above, proves Theorems \ref{thm1} and \ref{thm2}. 
    

       \section{Proof of Corollary \ref{cor1}}\label{cor}
 
    From Theorem \ref{thm2}, if $r\ge 2$, there exists $\delta>0$ such that $E_{N}^{(i)*}\ll N^{2.382-\delta}$ for all $1\le i\le r$, and if we consider $r\ge 3$, then $E_{N}^{(i)*}\ll N^{2.6367-\delta}$ for all $1\le i\le r$. By Lemma 5.2 of \cite{bkw10}, if $a^{i}(x_n^{i})$  is a sequence of quadratic polynomials with real coefficients, then for any $\epsilon>0$, $E_{N}^{(i)*}\ll N^{2+\epsilon}$. A sequence $a^{i}(x_n^{i})$ is said to be a lacunary sequence, if for all $n$ \[\frac{a^{i}(x_{n+1}^{i})}{a^{i}(x_n^{i})}\ge \lambda >1.\numberthis\label{lacu}\]
    For a lacunary sequence $E_{N}^{(i)*}\ll N^{2+\epsilon}$, see \cite[Proposition 4.2]{rudtec}. Moreover, if $a^{i}(x_n^{i})$ is said to be a convex sequence, that is, for all $n$, $a^{i}(x_n^{i})-a^{i}(x_{n-1}^{i})<a^{i}(x_{n+1}^{i})-a^{i}(x_n^{i})$, then it is shown in \cite{shk13} that for a convex sequence  $E_{N}^{(i)*}\ll N^{2.46+\epsilon}$ for any $\epsilon>0$.
  
       \section{Proof of Theorem \ref{thm3} }\label{thm3proof}
To prove Theorem \ref{thm3}, we follow the same method we used to prove Theorem \ref{thm2}  with some minor changes. So, we will merely sketch the details where there are differences between the two proofs. In section \ref{expectations}, we use the growth condition \eqref{growthcon} of the sequence to compute the expectation. Since we have the same growth condition in  Theorem \ref{thm3}, we obtain the same expression for the expectation of the pair correlation function here as well. Next, for the variance, we follow section \ref{variance} till \eqref{varsum}. This leads us to the sum \[\prod_{i=1}^r\sum_{\substack{1\le u_i,v_i\le M\\ \text{for some  i}\ u_i\ne v_i}}\sum_{2^{u-1}\le j_1,j_2<2^u}\I(|j_1\rho_i(u_i)-j_2\rho_i(v_i)|<1).\numberthis\label{varsumthm3}\]
      First, suppose that $\max\{\rho_i(u_i),\rho_i(v_i)\}\le 4N^{1/4}$ for all $1\le i\le r$. We follow the same argument given in \eqref{case21} and obtain the same bound in this case. Next, for some $1/4\le \beta\le 1.01$, we consider the sum 
      \[\sum_{2^{u-1}\le j_1,j_2<2^u}\prod_{i=1}^r\sum_{\substack{1\le u_i,v_i\le M\\\rho_i(v_i), \rho_i(u_i)\in [N^{\beta},32N^{\beta}) }}\I(|j_1\rho_i(u_i)-j_2\rho_i(v_i)|<1).\numberthis\label{case2countthm3}\]
 Now, define \[b_3(k_i):=\sum_{\substack{1\le u_i\le M\\\rho_i(u_i)\in [N^{\beta},32N^{\beta}) }}\I\left(\rho_i(u_i)\in\left[\frac{k_i}{2^u},\frac{k_i+1}{2^u}\right)\right).\]
       Set $T=2^{u/r}N^{\beta}$ and corresponding to $b_3(k_i)$, define $a(h_i)$, $P(s_i)$ similarly as in \eqref{ahdef} and \eqref{psdef}, respectively. Here, we take the solution on a finer scale by assuming $\gamma=2^{-u/r}$. From Lemma \ref{lemma2}, we have
       \[\prod_{i=1}^{r}\int_{\R}|P(s_i)|^2\Phi(s_i/T)ds_i\ll T^r({E_{N,2^{-u/r}}^*})^r.\]
       We get similar Lemmas \ref{lemma3} and \ref{lemma4} assuming $\rho_i(u_i)\in \left[\frac{k_i}{2^u},\frac{k_i+1}{2^u}\right)$ for this case. Trivially, $|P(0)|\ll N^2$, but in this case, we can find a better bound for $|P(0)|$ given by
       \begin{align*}
           |P(0)|&=\sum_{h_i\ge 0}a(h_i)\le \sum_{k_i}b(k_i)=\#\{u_i:\rho_i(u_i)\in [N^{\beta},32N^{\beta})  \}\\
           &\ll N^{\beta/2}\left(\sum_{a=N^{\beta}}^{32N^{\beta}}(\#\{u_i:\rho_i(u_i)\in [a,a+1)  \})^2\right)^{1/2} \ll N^{\beta/2}\sqrt{E^{*}_{N,1}}.\numberthis\label{newpbound}
       \end{align*}
       Next, we find an analog of Lemma \ref{lemma5} for this case. From \eqref{newpbound}, we have 
       \[|I_2|,|I_3|\ll N^{r\beta/2+1+\epsilon/4}(\sqrt{E^{*}_{N,1}})^r.\numberthis\label{i23bound}\]
       Using the bound $\zeta(1/2+\iota t)\ll |t|^{1/6}$ and $K(u)\ll u^{-2}$, we have
       \[I_{1}\ll T^{r+r/3}({E_{N,2^{-u/r}}^*})^r.\numberthis\label{i1bound}\]
       From \eqref{i23bound} and \eqref{i1bound}, we obtain 
      \begin{align*}
          &\prod_{i=1}^r\sum_{\substack{1\le u_i,v_i\le M\\ \text{for some  i}\ u_i\ne v_i}}\sum_{2^{u-1}\le j_1,j_2<2^u}\I(|j_1\rho_i(u_i)-j_2\rho_i(v_i)|<1)\\&
          \ll \frac{2^u}{T^r}\left(N^{r\beta+1+\epsilon/4}({E^{*}_{N,1}})^r+T^{r+r/3}({E_{N,2^{-u/r}}^*})^r\right)\\
          &\ll \left(N^{1+\epsilon/4}({E^{*}_{N,1}})^r+2^uT^{r/3}({E_{N,2^{-u/r}}^*})^r\right)\\&\ll N^{4r-\epsilon},\numberthis\label{thminequal1}
      \end{align*}
      where the last inequality comes from \eqref{thm3bound} for some $\epsilon>0$. For  $\min\{\rho_i(u_i),\rho_i(v_i)\}>N^{1.01}$ for all $1\le i\le r$, we choose $T=2^{u/r}N^{1+\epsilon/2}$ and we can not obtain finer bound of $P|0|$, so we use $|P(0)|\ll N^2$. This yields
      \begin{align*}
          &\prod_{i=1}^r\sum_{\substack{1\le u_i,v_i\le M\\ \text{for some  i}\ u_i\ne v_i}}\sum_{2^{u-1}\le j_1,j_2<2^u}\I(|j_1\rho_i(u_i)-j_2\rho_i(v_i)|<1)\\&
          \ll \frac{2^u}{T^r}\left(N^{4r+1+\epsilon/4}+T^{r+r/3}({E_{N,2^{-u/r}}^*})^r\right)\ll \left(N^{3r+1+\epsilon/4}+T^{r/3}({E_{N,2^{-u/r}}^*})^r\right)\\ &\ll N^{4r-\epsilon},\numberthis\label{thminequal2}
      \end{align*}
      where the last inequality comes from \eqref{thm3bound} for some $\epsilon>0$. If for some $k\ge 1$ and  $1\le i\le k$,  $\min\{\rho_i(u_i),\rho_i(v_i)\}>N^{1.01}$  and $\min\{\rho_i(u_i),\rho_i(v_i)\}\le N^{1.01}$ for $k+1\le i\le r$ . Then, from  \eqref{thminequal1} and \eqref{thminequal2}, we have \begin{align*}
       &\prod_{i=1}^r\sum_{\substack{1\le u_i,v_i\le M\\ \text{for some  i}\ u_i\ne v_i}}\sum_{2^{u-1}\le j_1,j_2<2^u}\I(|j_1\rho_i(u_i)-j_2\rho_i(v_i)|<1)\\&\le \prod_{i=1}^k\sum_{\substack{1\le u_i,v_i\le M\\ \text{for some  i}\ u_i\ne v_i}}\sum_{2^{u-1}\le j_1,j_2<2^u}\I(|j_1\rho_i(u_i)-j_2\rho_i(v_i)|<1)\\&\times\prod_{l=1}^{r-k}\sum_{\substack{1\le u_l,v_l\le M\\ \text{for some  l}\ u_l\ne v_l}}\sum_{2^{u-1}\le j_1,j_2<2^u}\I(|j_1\rho_l(u_l)-j_2\rho_l(v_l)|<1)\\&\ll N^{4r-\epsilon}.
   \end{align*}
   Next, we follow the arguments in section \ref{mainproof} to obtain the required result.
\bibliographystyle{amsalpha}
\bibliography{ref} 
    \end{document}